\documentclass[conference,a4paper]{IEEEtran}

\PassOptionsToPackage{cmex10}{amsmath}
\usepackage{style}
\usepackage[utf8]{inputenc} 
\usepackage[T1]{fontenc}
\usepackage{url}              
\usepackage{cite}             

\usepackage{amssymb,amsbsy,amsthm,dsfont}
\usepackage{mleftright}       
\mleftright                   

\usepackage{graphicx}         
\usepackage{booktabs}         

\begin{document}

\title{Matching Correlated Inhomogeneous \\ Random Graphs using the $k$-core Estimator} 

\author{%
  \IEEEauthorblockN{Mikl\'{o}s Z. R\'{a}cz}
  \IEEEauthorblockA{%
   Department of Statistics and Data Science \\
   Department of Computer Science \\
    Northwestern University \\
    Evanston, IL \\
    miklos.racz@northwestern.edu}
    \and 
    \IEEEauthorblockN{Anirudh Sridhar}
    \IEEEauthorblockA{%
    Department of Electrical and Computer Engineering \\
    Princeton University \\
    Princeton, NJ \\
    anirudhs@princeton.edu}
    }

\maketitle

\begin{abstract}
We consider the task of estimating the latent vertex correspondence between two edge-correlated random graphs with generic, inhomogeneous structure. We study the so-called \emph{$k$-core estimator}, which outputs a vertex correspondence that induces a large, common subgraph of both graphs which has minimum degree at least $k$. We derive sufficient conditions under which the $k$-core estimator exactly or partially recovers the latent vertex correspondence. Finally, we specialize our general framework to derive new results on exact and partial recovery in correlated stochastic block models, correlated Chung-Lu graphs, and correlated random geometric graphs.
\end{abstract}

\section{Introduction}

In the past decade there has been a strong and growing interest in understanding the fundamental limits of graph matching, both information-theoretically~\cite{pedarsani2011privacy,cullina2016improved,cullina2018exact,cullina2020partial,wu2022settling, ganassali2021impossibility,ding2022matching, hall2020partial, onaran2016optimal,RS21,RS22,shirani2021concentration} and algorithmically~\cite{yartseva2013percolation,shirani2017seeded,barak2019,ding2021efficient,fan2020spectral, mossel2019seeded, ganassali2020tree,
mao2021random, mao2021exact, mao2022otter}, 
leading to several recent breakthroughs. 
Most of the focus thus far has been on the simplest setting of correlated Erd\H{o}s--R\'enyi random graphs, 
with only a few works studying models beyond this, 
such as correlated stochastic block models~\cite{onaran2016optimal,RS21,GRS22}, correlated randomly growing graphs~\cite{korula2014efficient,RS22}, correlated Chung--Lu graphs~\cite{yu2021power, chiasserini2016social, bringmann2014heterogeneous}, and correlated geometric models~\cite{wang2022geometric}.

\subsection{Contributions}

In this work, we initiate a systematic development of techniques for matching correlated networks with general, inhomogeneous structure. 
The matching procedure we study, known as the \emph{$k$-core estimator}, produces a (possibly partial) matching of two graphs, $G_1$ and $G_2$, that induces a large, common subgraph of minimum degree at least~$k$.
Building on the analyses of~\cite{cullina2020partial, GRS22}, we show that, for an appropriately chosen value of $k$ based on the density of the graphs, the $k$-core estimator is guaranteed to produce a \emph{fully correct} (possibly partial) matching of the two graphs. 

We then derive general conditions under which the $k$-core estimator fully or partially recovers the latent matching. 
For both objectives, we show that the success of the $k$-core estimator depends on the \emph{minimum expected degree} of the intersection graph of $G_1$ and $G_2$.
In particular, exact recovery of the matching is guaranteed when the minimum expected degree is larger than $\log n$, where $n$ is the number of vertices in the graph.
This condition is known to be the information-theoretic threshold for exact graph matching in Erd\H{o}s--R\'{e}nyi graphs~\cite{cullina2016improved,cullina2018exact} and stochastic block models with two balanced communities~\cite{RS21}; we conjecture that it may be information-theoretically optimal for a much broader range of models as well.
For the goal of partial recovery, we show that the $k$-core estimator succeeds when the minimum expected degree of the intersection graph is larger than a function of the \emph{inhomogeneity} of the graphs.

Finally, we apply our general results to several well-studied generative graph models of interest.
Specifically, our work provides the first theoretical analysis of seedless graph matching for general stochastic block models, Chung-Lu models, and noisy random geometric graphs. 
These results highlight the power of the $k$-core estimator as a useful, model-agnostic procedure for understanding graph matching in more realistic network models.

\subsection{Related work}

In recent years, there have been significant developments in the study of information-theoretic characterizations of graph matching. 
In the context of correlated Erd\H{o}s--R\'{e}nyi graphs, the information-theoretic conditions for the possiblity and impossibility of exact graph matching were derived in~\cite{cullina2016improved, cullina2018exact, wu2022settling}. 
Several works have also addressed information-theoretic characterizations of \emph{partial} graph matching (i.e., where the goal is to correctly recover any positive fraction of the ground-truth matching)~\cite{cullina2020partial, ding2022matching, wu2022settling, hall2020partial, ganassali2020tree}, with impossibility results established in \cite{wu2022settling, ganassali2021impossibility} and matching achievability results very recently determined by Ding and Du~\cite{ding2022matching}.
Moreover, we remark that the $k$-core estimator we use for graph matching in this paper is closely related to the \emph{dense subgraph estimator} used in~\cite{ding2022matching}, which is information-theoretically optimal for partial recovery. See Remark \ref{remark:densest_subgraph} for more details on this point. We also remark that while our paper, as well as the above literature, largely focuses on information-theoretic conditions for graph matching, a fascinating related area of research is to design \emph{efficient algorithms} that operate in the achievability region (see, e.g., \cite{ganassali2020tree, barak2019, ding2021efficient, fan2020spectral, mao2021exact, mao2021random, mao2022otter}).

To the best of our knowledge, little is known of the information-theoretic limits of graph matching for networks beyond Erd\H{o}s--R\'{e}nyi graphs. 
A model of correlated stochastic block models was first studied by Onaran, Garg, and Erkip~\cite{onaran2016optimal}, and R\'{a}cz and Sridhar~\cite{RS21} later determined the information-theoretic limits of exact recovery in this setting.
Recently, Wang, Wu, Xu, and Yolou~\cite{wang2022geometric} studied exact and almost exact graph matching of two complete graphs with edge weights corresponding to latent geometric structure.
A common thread of these works is that the statistic used to recover the matching (e.g., the maximum a posteriori (MAP) estimator), as well as the methods of analysis, are highly sensitive to the type of network model considered.
In contrast, our work provides a \emph{model-agnostic} toolkit for deriving achievability results for exact and partial graph matching.

We comment on a related but different line of work on graph matching with side information in the form of many correctly matched ``seeds" that are known a priori.
In this setting, there are numerous efficient algorithms with provable guarantees for exactly recovering the latent matching in correlated Erd\H{o}s-R\'{e}nyi graphs~\cite{mossel2019seeded, pedarsani2011privacy, yartseva2013percolation, shirani2017seeded, kazemi2015growing}, correlated Chung-Lu graphs~\cite{yu2021power, chiasserini2016social, bringmann2014heterogeneous} and correlated preferential attachment graphs~\cite{korula2014efficient}.
Our work, on the other hand, studies what can be learned \emph{without} such side information.

\subsection{Outline}

The rest of this paper is organized as follows: 
we first present the models we study in Section~\ref{sec:models} 
and subsequently detail our results in Section~\ref{sec:results}. 
The proofs can be found in Section~\ref{sec:proofs}, with some proofs deferred to the appendices. 
We conclude with a brief discussion in Section~\ref{sec:discussion}.

\subsection{Notation}

We represent a graph as $G = (V,E)$, where $V$ is the vertex set and $E$ is the set of edges. For a vertex $v \in V$, we let $\deg_{G}(v)$ denote its degree in $G$. 
We let $\core_k(G)$ denote the $k$-core of $G$, which is the maximal subgraph with minimum degree $k$. For a set $S \subseteq V$, $G \{ S \}$ represents the induced subgraph of $G$ corresponding to $S$.
For random variables $X$ and $Y$ we write $X \preceq Y$ to denote that $X$ is stochastically dominated by $Y$.
Throughout the paper, we use standard asymptotic notation (e.g., $O(\cdot), o(\cdot)$).

\section{Correlated inhomogeneous random graphs}\label{sec:models}

We start by describing a general model of inhomogeneous random graphs. 

\begin{definition}[Inhomogeneous random graphs]\label{def:irg}
Let $n$ be a positive integer, and let $\mathbf{p} : = \{ p_{ij} \}_{i,j \in [n] }$ be a collection of elements in $[0,1]$ representing edge probabilities. We construct $G \sim \inhomogeneous(\mathbf{p} )$ by adding the edge $(i,j)$ to 
$G$ 
with probability $p_{ij}$, independently across all pairs of elements of $[n]$.
\end{definition}

This is a very well-studied general model (see, e.g., the book~\cite{frieze2016introduction}). 
Many well-known probabilistic generative models for networks can be viewed as special cases of the inhomogeneous random graph model described above. 
For instance, if $p_{ij} = p$ for all distinct $i,j \in [n]$, we recover the Erd\H{o}s-R\'{e}nyi 
graph---perhaps the most basic generative network model. 

In more complex models, the $p_{ij}$'s may be induced by \emph{latent structure} which can dramatically shape the topology of the resulting random graph. 
That is, to each vertex $i \in [n]$ there is an associated latent variable (e.g., community membership, weight, latent position), 
often drawn at random from some distribution, 
and \emph{conditioned on these latent variables} 
the graph is an inhomogeneous random graph as in Definition~\ref{def:irg}, 
with the $p_{ij}$'s being a function of the latent variables. 

In what follows we describe three commonly studied models of this type. 
Our general results (see Theorems~\ref{thm:exact_recovery} and~\ref{thm:partial_matching}) apply to inhomogeneous random graphs as in Definition~\ref{def:irg}; 
consequently, they apply to the specific models \emph{conditioned on the latent variables}. 
Because of this, we present these models already conditioned on the latent variables, without specifying the distribution over the latent variables. 
However, 
as we shall see, the conditions of these theorems are \emph{simple}, 
which means that it is not difficult 
to understand when they hold with high probability over the distribution of the latent variables. We leave these details 
to 
the reader.

\begin{example}
In stochastic block models (SBMs)~\cite{Abbe_survey,HLL83}, the $p_{ij}$'s are induced by latent community structure. 
Specifically, suppose that there are $m$ communities in the network, represented by a partition $V_1, \ldots, V_m$ of $[n]$. 
We also specify a collection $\mathbf{q} = \{ q_{ab} \}_{a,b \in [m]}$ of $[0,1]$-valued elements representing edge formation probabilities within and across communities. We say that $G \sim \sbm(V_1, \ldots, V_m, \mathbf{q})$ if $p_{ij} = q_{ab}$ for distinct $i,j \in [n]$ whenever $i \in V_a$ and $j \in V_b$.
\end{example}

\begin{example}
In the Chung-Lu model~\cite{chung2002average}, the $p_{ij}$'s are induced by latent vertex weights. Let $\mathbf{w} \in \mathbb{R}_+^n$ be a collection of vertex weights satisfying $\max_{i \in [n]} w_i \le \sqrt{ \sum_{i \in [n]} w_i}$. We say that $G \sim \cl( \mathbf{w})$ if $p_{ij} = w_i w_j / ( \sum_{k \in [n]} w_k )$ for distinct $i,j \in [n]$.
\end{example}

\begin{example}
Consider a noisy random geometric graph (see~\cite{liu2021phase} and the references therein), constructed as follows. Let $d$ be a positive integer (which could be constant or increasing with $n$), and let $\mathbf{x} := \{ x_1, \ldots, x_n \}$ be 
elements of the $d$-dimensional sphere $\mathbb{S}^{d-1}$. 
Given parameters $r \in [0,2]$ and $p \in [0,1]$, we say that $G \sim \rg( \mathbf{x}, r, p)$ if 
\begin{equation}
\label{eq:rg_probability} 
p_{ij} = p \mathbf{1}( \| x_i - x_j \|_2 \le r )
\end{equation}
for distinct $i,j \in [n]$.
\end{example}

We next define correlated inhomogeneous random graphs. 

\begin{definition}[Correlated inhomogeneous random graphs]
Let $n$ be a positive integer, let $\mathbf{p}$ be a collection of edge probabilities, and let $s \in [0,1]$ be a correlation parameter. A pair $(G_1, G_2)$ of correlated inhomogeneous random graphs are generated according to the following process. First, a parent $G \sim \inhomogeneous ( \mathbf{p} )$ is generated. Conditioned on $G$, a pair of graphs $(G_1, G_2')$ with the same vertex sets as $G$ are independently constructed by subsampling each edge in $G$ with probability~$s$. Finally, a uniform random permutation $\pi_*$ is applied to the vertex labels of $G_2'$ to generate $G_2$.
For brevity, we say that $(G_1, G_2) \sim \correlated(\mathbf{p}, s)$ if the pair $(G_1, G_2)$ has been constructed in this manner.
\end{definition}

Correspondingly, we may also construct correlated stochastic block models, Chung-Lu graphs, and noisy random geometric graphs. We refer to these distributions as $\csbm, \ccl$ and $\crg$, respectively.

\section{Results: graph matching via the $k$-core estimator}\label{sec:results}

In this section, we 
present our main 
results establishing sufficient conditions for the correctness of the $k$-core estimator for correlated inhomogeneous random graphs.

\subsection{Correctness of the $k$-core estimator}

Let us start with some definitions.
\begin{definition}[Matching]
A pair $(M, \mu)$ is a \emph{matching} between $G_1$ and $G_2$ if $M \subseteq [n]$, $\mu: M \to [n]$, and $\mu$ is injective.
\end{definition}

We write $G_1 \land_{\mu} G_2$ to be the \emph{intersection graph} with respect to the matching $(M,\mu)$. Specifically, $G_1 \land_{\mu} G_2$ has a vertex set equal to $M$ and $(i,j)$ is an edge in $G_1 \land_{\mu} G_2$ if and only if $(i,j)$ is an edge in $G_1$ and $(\mu(i), \mu(j))$ is an edge in $G_2$.

\begin{definition}[$k$-core matching]
A matching $(M, \mu)$ is a \emph{$k$-core matching} if the minimum degree in $G_1 \land_{\mu} G_2$ is at least~$k$. 
\end{definition}

This terminology comes from the notion of a \emph{$k$-core}: the $k$-core of a graph is the maximal subgraph with minimum degree $k$. 

\begin{definition}[$k$-core estimator]
Given a pair of graphs $(G_{1},G_{2})$ on $n$ vertices and $k \in [n]$, 
the \emph{$k$-core estimator} $(\widehat{M}_{k},\widehat{\mu}_{k})$ 
is a $k$-core matching that involves the greatest number of vertices among all $k$-core matchings. (If this is not uniquely defined, pick an arbitrary such matching.)
\end{definition}

The following lemma provides a general and simple sufficient condition under which the $k$-core estimator succeeds with high probability. More precisely, the conclusion of the lemma guarantees that the $k$-core estimator exactly recovers the latent matching on the $k$-core of the intersection graph, making no errors (and it makes no matches outside of this $k$-core). 

Before stating the result, we define, for a matching $(M, \mu)$ and $S \subseteq M$, the set $\mu \{ S \} : = \{ (i, \mu(i) ) : i \in S \}$.

\begin{lemma}
\label{lemma:k_core_correctness} 
Let $(G_1, G_2) \sim \correlated( \mathbf{p} , s)$.  
Suppose that $p_{\max} \le n^{- \alpha + o(1)}$ for some $\alpha \in (1/2, 1]$ and that $k > 12/ ( 2 \alpha - 1)$. Then, with probability $1 - o(1)$, it holds that
\[
\widehat{M}_k = \core_k ( G_1 \land_{\pi_*} G_2 ) \text{ and } \widehat{\mu}_k \{ \widehat{M}_k \} = \pi_* \{ \widehat{M}_k \}.
\]
\end{lemma}
The proof of this lemma can be found in Appendix~\ref{sec:proof_k-core_lemma}.

\begin{remark}
\label{remark:densest_subgraph}
The $k$-core estimator bears some similarity to the \emph{dense subgraph estimator} of \cite{ding2022matching}, which outputs an estimator $\widehat{\mu}$ for which $G_1 \land_{\widehat{\mu}} G_2$ contains a large subgraph with a high \emph{average} degree. 
Remarkably, the dense subgraph estimator succeeds in partial graph matching all the way down to the information-theoretic threshold. 

In a sense, the $k$-core estimator can be viewed as a \emph{robust} version of the dense subgraph estimator: by finding a subgraph with a lower bound on the \emph{minimum} degree (rather than on the average degree), we can guarantee that all vertices in the matching are correctly matched. On the other hand, the dense subgraph estimator outputs a full matching rather than a partial matching, which makes several errors in sparse regimes. However, it is an interesting open problem whether the guarantees of the $k$-core estimator can be extended in some sense to the dense subgraph estimator.
\end{remark}

\subsection{Exact graph matching}

Lemma~\ref{lemma:k_core_correctness} allows to obtain sufficient conditions under which the $k$-core estimator $\widehat{\mu}_{k}$ recovers $\pi_{*}$, either exactly or partially. 
The next result is our main general theorem about exact graph matching in correlated inhomogeneous random graphs. 

\begin{theorem}
\label{thm:exact_recovery} 
Let $(G_1, G_2) \sim \correlated( \mathbf{p} , s)$.  
Suppose that $p_{\max} \le n^{- \alpha + o(1)}$ for some $\alpha \in (1/2, 1]$ and that $k > 12/ ( 2 \alpha - 1)$ is a constant. 
Furthermore, suppose that there exists $\epsilon > 0$ such that 
\begin{equation}\label{eq:expected_degree_condition} 
\min_{i \in [n]} \sum_{j=1}^{n} p_{ij} s^{2} \ge (1 + \epsilon) \log n.
\end{equation}
Then $\p ( \widehat{\mu}_k = \pi_* ) = 1 - o(1)$ as $n \to \infty$.
\end{theorem}

That is, under the conditions of the theorem, the $k$-core estimator $\widehat{\mu}_{k}$ exactly recovers $\pi_{*}$ with high probability. 
The condition~\eqref{eq:expected_degree_condition} is simple and natural: in words, it says that the expected degree of every node in the intersection graph $G_{1} \land_{\pi_{*}} G_{2}$ is at least $(1+\epsilon) \log n$. 
In the Erd\H{o}s--R\'enyi case this simplifies to $n p s^{2} \geq (1+\epsilon) \log n$, 
which is known to be information-theoretically optimal~\cite{cullina2016improved,cullina2018exact}. 
More generally, depending on $\mathbf{p}$, it may be possible to improve upon~\eqref{eq:expected_degree_condition}; in fact, we prove a slightly stronger result (with a weaker sufficient condition) in Theorem~\ref{thm:general_exact_recovery} below. 
However, we conjecture that the simple sufficient condition in Theorem~\ref{thm:exact_recovery} is near-optimal for a wide range of~$\mathbf{p}$. 

We next specialize Theorem~\ref{thm:exact_recovery} to three important cases: correlated stochastic block models, correlated Chung-Lu graphs, and correlated random geometric graphs. 
The proofs are short, and can be found in Appendix \ref{sec:proofs_corollaries}.
We note that in all three cases the sufficient condition for exact recovery is a condition on the minimum expected degree, which is simple to compute for most natural distributions over the latent variables.

\begin{corollary}\label{cor:sbm_exact}
Let $(G_1, G_2) \sim \csbm( V_1, \ldots, V_m, \mathbf{q}, s)$.
Suppose that $\max_{a,b \in [m]} q_{ab} \le n^{- \alpha + o(1)}$ for some $\alpha \in (1/2, 1]$ and  that $k > 12/ ( 2 \alpha - 1)$ is a constant.
Furthermore, suppose that there exists $\epsilon > 0$ such that  
\[
\min_{a \in [m]} \sum_{b = 1}^m |V_b| q_{ab} s^2 \ge (1 + \epsilon) \log n.
\]
Then $\p ( \widehat{\mu}_k = \pi_* ) = 1 - o(1)$ as $n \to \infty$.
\end{corollary}


\begin{corollary}\label{cor:cl_exact}
Let $(G_1, G_2) \sim \ccl( \mathbf{w}, s)$.
Suppose that 
\begin{equation}
\label{eq:chung_lu_exact_recovery_1}
\max_{i \in [n]} w_i \le n^{ - \alpha / 2 + o(1)} \sqrt{  \sum_{i = 1}^n w_i },
\end{equation}
for some $\alpha \in (1/2, 1]$, and let $k > 12/(2 \alpha - 1)$ be a constant.
Furthermore, suppose that there exists $\epsilon > 0$ such that
\[
\min_{i \in [n]} w_i s^2 \ge (1 + \epsilon) \log n.
\]
Then $\p( \widehat{\mu}_k = \pi_* ) = 1 - o(1)$ as $n \to \infty$.
\end{corollary}


\begin{corollary}\label{cor:rgg_exact}
Let $(G_1, G_2) \sim \crg(\mathbf{x}, r, p)$. Suppose that $p \le n^{- \alpha + o(1)}$ for some $\alpha \in (1/2, 1]$, and let $k > 12/(2 \alpha - 1)$ be a constant.
Furthermore, suppose that there exists $\epsilon > 0$ such that 
\[
p s^2 \min_{i \in [n]} | \{ j \in [n]: \| x_i - x_j \| \le r \} |  \ge (1 + \epsilon ) \log n.
\]
Then $\p ( \widehat{\mu}_k = \pi_* ) = 1 - o(1)$ as $n \to \infty$.
\end{corollary}


\subsection{Partial graph matching}

Going beyond exact matching, Lemma~\ref{lemma:k_core_correctness} also allows us to obtain quantitative results on partially recovering $\pi_{*}$ using the $k$-core estimator $\widehat{\mu}_{k}$, when exact recovery is not possible. 
Before stating our main results, we define
\[
R : = \frac{ (n-1) p_{\max} }{ \min_{i \in [n]} \sum_{j = 1}^n p_{ij}}.
\]
The quantity $R$ measures, in a sense, the \emph{heterogeneity} in the inhomogeneous graph model.
Indeed, if the $p_{ij}$'s were constant, then $R = 1$. However, if there are significant differences between the minimum and maximum values of $p_{ij}$, then $R$ takes on a larger value.

The next result is our main general theorem about partial graph matching in correlated inhomogeneous random graphs. It highlights the role of $R$ in the partial recovery of $\pi_*$. 
For simplicity we focus here on the case when $p_{\max} = n^{-1+o(1)}$, 
though 
the techniques extend to a wider range of~$\mathbf{p}$. 

\begin{theorem}
\label{thm:partial_matching} 
Let $(G_1, G_2) \sim \correlated( \mathbf{p} , s)$.  
Suppose that $p_{\max} = n^{-1 + o(1)}$ and $k = 13$.
If 
\[
\min_{i \in [n]} \sum_{j = 1}^n p_{ij} s^2 \ge \max \{ R, 154 \},
\]
then with probability $1 - o(1)$ it holds that 
\begin{equation}\label{eq:fraction_matched}
 | \widehat{M}_k | \ge \left( 1  - 3e^{- \min_{i \in [n]} \sum_{j = 1}^n p_{ij} s^2 / 7} \right) n - o(n)
\end{equation}
and that 
\begin{equation}\label{eq:all_correctly_matched}
\widehat{\mu}_k \{ \widehat{M}_k \} = \pi_* \{ \widehat{M}_k \}. 
\end{equation}
\end{theorem}

That is, under the conditions of the theorem, 
the $k$-core estimator matches a 
$(1-3\exp(-\min_{i \in [n]} \sum_{j = 1}^n p_{ij} s^{2} / 7) - o(1))$ 
fraction of the nodes (see~\eqref{eq:fraction_matched})
and all of these matches are correct (see~\eqref{eq:all_correctly_matched}), with high probability. 
We emphasize that 
the property of not making any erroneous matches is very useful, 
especially when the matching is further used for downstream tasks 
(e.g., community recovery~\cite{GRS22}).

We again specialize the general theorem to the three cases.
The proofs are short, and can be found in Appendix \ref{sec:proofs_corollaries}.
Here, too, in all three cases, the sufficient condition for partial recovery of $\pi_{*}$ depends on the minimum expected degree, as well as $R$ (a measure of heterogeneity), 
which are simple to compute for most natural distributions over the latent variables. 

\begin{corollary}\label{cor:sbm_partial}
Let $(G_1, G_2) \sim \csbm( V_1, \ldots, V_m, \mathbf{q},s)$
and let $q_{\max} := \max_{a,b \in [m]} q_{ab}$ and $q_{\min} : = \min_{a,b \in [m]} q_{ab}$.
Suppose that $q_{\max} = n^{-1 + o(1)}$ and $k = 13$.
The result of Theorem \ref{thm:partial_matching} holds if there exists $\epsilon > 0$ such that
\[
\min_{a \in [m]} \sum_{b = 1}^m |V_b| q_{ab} s^2 \ge  \max \left  \{\frac{q_{\max} }{q_{\min}}, 154 \right \} + \epsilon .
\]
\end{corollary}


\begin{corollary}\label{cor:cl_partial}
Let $(G_1, G_2) \sim \ccl( \mathbf{w}, s)$,
and denote $w_{\max} : = \max_{i \in [n]} w_i$ and $w_{\min} : = \min_{i \in [n]} w_i$.
The result of Theorem \ref{thm:partial_matching} holds if $w_{\max}$ is bounded as $n \to \infty$, and there is $\epsilon > 0$ such that
\[
w_{\min} s^2 \ge \max \left \{  \left( \frac{w_{\max}}{w_{\min}} \right)^2 , 154 \right \} + \epsilon .
\]
\end{corollary}


\begin{corollary}\label{cor:rgg_partial}
Let $(G_1, G_2) \sim \crg(\mathbf{x}, r, p)$. The result of Theorem~\ref{thm:partial_matching} holds if $p = n^{-1 + o(1)}$ and it holds for some $\epsilon > 0$ that
\[
p s^{2} \min_{i \in [n]} | \{ j \in [n]: \| x_i - x_j \| \le r \} | 
\geq \max \{ s\sqrt{np} , 154 \} + \epsilon .
\]
\end{corollary}

\section{Proofs: Recovering the latent matching}\label{sec:proofs}

\subsection{Properties of the degree distribution}

For a graph $G$ and a positive integer $k$, define the set 
\begin{equation}
\label{eq:Zk}
Z_k : =  \{ i \in [n]: \deg_G(i) \le k \}.
\end{equation}
We 
state 
some useful results on the size of $Z_k$ in inhomogeneous random graphs. 
We defer their proofs to Appendix~\ref{sec:proofs_lemmas_degrees}. 

\begin{lemma}
\label{lemma:Zk_expectation}
Let $G \sim \inhomogeneous( \mathbf{p} )$. Then for any positive integer $k$ and any $c \in (0,1)$ we have that
\[
\E \left[ | Z_k | \right] \\
\le \left( \frac{1}{1 - c} \right)^k \sum_{i = 1}^n \exp \left ( - c \sum_{j = 1}^n p_{ij}  \right ).
\]
\end{lemma}


\begin{lemma}
\label{lemma:Zk_concentration}
Let $G \sim \inhomogeneous ( \mathbf{p} )$, and suppose that $p_{\max} : = \max_{i,j \in [n] } p_{ij} = o(1/\sqrt{n})$. Then for every $k \in [n]$, 
\[
\p \left( \left| |Z_k| - \E \left[ |Z_k| \right] \right| \le \frac{1}{3} n^{3/4} \right) = 1 - o(1), 
\quad 
\text{ as } n \to \infty.
\]
\end{lemma}

\subsection{Exact recovery}

We state and prove here a result that is slightly stronger than Theorem~\ref{thm:exact_recovery} (which then follows immediately, see below). 

\begin{theorem}
\label{thm:general_exact_recovery} 
Let $(G_1, G_2) \sim \correlated( \mathbf{p} , s)$.  
Suppose that $p_{\max} \le n^{-\alpha + o(1)}$ for some $\alpha \in (1/2,1]$ 
and 
that $k > 12/(2 \alpha - 1)$. Then for any $c \in (0,1)$,
\[
\p ( \widehat{\mu}_k \neq \pi_* ) \le \left( \frac{1}{1 - c} \right)^{k-1} \sum_{i = 1}^n \exp \left( - c \sum_{j = 1}^n p_{ij} s^2 \right) + o(1),
\]
where $o(1) \to 0$ as  $n \to \infty$.
\end{theorem}

\begin{proof}
We can bound the probability of interest as follows:
\begin{align}
& \p ( \widehat{\mu}_k \neq \pi_* ) \le \p ( \core_k ( G_1 \land_{\pi_*} G_2 ) \neq [n] ) \nonumber  \\
& \hspace{0.3cm} + \p \left ( \widehat{M}_k \neq \core_k( G_1 \land_{\pi_*} G_2) \text{ or } \widehat{\mu}_k \{ \widehat{M}_k \} \neq \pi_* \{ \widehat{M}_k \} \right ) \nonumber \\
& \stackrel{(a)}{\le} \p ( \mindeg(G_1 \land_{\pi_*} G_2 ) < k ) + o(1) \nonumber \\
& \stackrel{(b)}{\le} \sum_{i = 1}^n \p ( \deg_{G_1 \land_{\pi_*} G_2} (i) < k ) + o(1)  \nonumber \\
\label{eq:exact_recovery_probability_bound}
& = \E \left[ | \{ i \in [n] : \deg_{G_1 \land_{\pi_*} G_2}(i) \le k - 1 \} | \right ] + o(1).
\end{align}
Above, $(a)$ follows from an application of Lemma \ref{lemma:k_core_correctness}, together with the observation that if the $k$-core of a graph does not encompass the entire vertex set, then the minimum degree of the graph must be less than $k$.
The inequality $(b)$ follows from a union bound.
To conclude, we use Lemma \ref{lemma:Zk_expectation} to bound the expectation in~\eqref{eq:exact_recovery_probability_bound}.
\end{proof}


\begin{proof}[Proof of Theorem~\ref{thm:exact_recovery}]
Set $c = ( 1 + \epsilon)^{-1/2}$ in Theorem~\ref{thm:general_exact_recovery}. 
Invoking the 
assumption that for every $i \in [n]$ we have the lower bound
$\sum_{j = 1}^n p_{ij} s^2 \geq (1+\epsilon) \log n$, this shows that
\[ 
\p ( \widehat{\mu}_k \neq \pi_* ) \le \left( \frac{1}{1 - (1+\epsilon)^{-1/2}} \right)^{k-1} n^{1-\sqrt{1+\epsilon}} + o(1).
\]
This bound vanishes as $n \to \infty$ for $k$ fixed.
\end{proof}

\subsection{Partial recovery}

We start by stating a useful result of \Luczak concerning the density of small, induced subgraphs of Erd\H{o}s-R\'{e}nyi graphs. 

\begin{lemma}[\cite{Luczak1991}]
\label{lemma:luczak}
Suppose that $\gamma = o ( \sqrt{n})$ and let $G \sim \er(n, \gamma / n)$. Then, 
with probability $1 - o(1)$, 
for every $S \subset [n]$ with $|S| \le \frac{3}{4 \gamma^2} n$, 
$G \{ S \}$ has at most $2 | S |$ edges.
\end{lemma}

Through a simple coupling argument, we can extend this lemma to the context of inhomogeneous random graphs.

\begin{lemma}
\label{lemma:subgraph_density}
Suppose that $p_{\max} \le \gamma / n$, where $\gamma = o ( \sqrt{n})$. Then the conclusion of Lemma \ref{lemma:luczak} holds for $G \sim \cG( \mathbf{p} )$.
\end{lemma}

\begin{proof}
Suppose we couple $G$ with $G' \sim \cG(n, \gamma / n)$ so that $G$ is always a subgraph of $G'$. Hence, for any $S \subset [n]$ with $|S| \le \frac{3}{4 \gamma^2} n$, if $G' \{S \}$ has at most $2 |S|$ edges, then $G \{ S \}$ has at most $2 |S|$ edges as well. The desired result follows. 
\end{proof}

Let $F_k := [n] \setminus \core_k ( G)$ be the set of vertices outside the $k$-core of $G$. Recall also the definition of $Z_{k}$ from~\eqref{eq:Zk}; whenever we use $Z_{k}$ in the following, the underlying graph will be clear from context.

\begin{lemma}
\label{lemma:k_core_size_1} 
Let $G \sim \cG( \mathbf{p} )$. 
Suppose that $p_{ij} \le \gamma / n$ for all $i,j \in [n]$, with $\gamma = o (\sqrt{n})$. Suppose that the event in Lemma \ref{lemma:subgraph_density} holds and that 
$| Z_{k+1} | \le  n/ (4 \gamma^2)$. Then $|F_k| \le 3 |Z_{k + 1}|$.
\end{lemma}

The proof follows using similar methods as \Luczak~\cite{Luczak1991}; we defer it to Appendix~\ref{sec:Luczak}.




\begin{lemma}
\label{lemma:k_core_expectation_bound}
Assume that $p_{\max} = o(n^{-7/8})$; that is, assume that $\gamma = o(n^{1/8})$. Assume also that $ \E \left [ | Z_{k+1} | \right ] \le n/(5 \gamma^2)$.
Then, with probability $1 - o(1)$, it holds that 
\[
|F_{k}| 
\le 3\E \left [ | Z_{k+1} | \right ] + n^{3/4}.
\]
\end{lemma}

\begin{proof}
Lemma~\ref{lemma:Zk_concentration} implies that, 
with probability $1 - o(1)$, 
we have that $3 |Z_{k+1} | \le 3 \E [ |Z_{k+1}| ] + n^{3/4}$, and 
\[
|Z_{k+1}| 
\le \E [ |Z_{k+1} |] + \frac{n^{3/4}}{3} 
\le \frac{n}{5 \gamma^2} + \frac{n^{3/4}}{3} 
\le \frac{n}{4 \gamma^2},
\]
where we have used that $\gamma = o ( n^{1/8})$ in the final inequality. The desired result follows from applying Lemma~\ref{lemma:k_core_size_1}.
\end{proof}

We are now ready to prove the main result on partial graph matching.

\begin{proof}[Proof of Theorem \ref{thm:partial_matching}]
As a shorthand, define $\dmin : = \min_{i \in [n]} \sum_{j = 1}^n p_{ij}$.
By Lemma~\ref{lemma:k_core_correctness}, 
it suffices to show that 
\begin{equation}
\label{eq:core_lower_bound}
|\core_{k} ( G_{1} \land_{\pi_{*}} G_{2} )| 
\geq 
\left( 1  - 3e^{- \dmin s^2 / 7} \right) n - o(n)
\end{equation}
holds with high probability. 
By Lemma~\ref{lemma:k_core_expectation_bound} 
it thus suffices to bound the expected number of vertices with degree at most $k+1=14$ in $G_1 \land_{\pi_*} G_2$. 
By setting $c = 1/2$ in Lemma \ref{lemma:Zk_expectation}, this is at most 
\begin{multline*}
2^{k + 1}  \sum_{i = 1}^n \exp \left( - \frac{1}{2} \sum_{j = 1}^n p_{ij} s^2 \right)  \le n 2^{k + 1} \exp \left( - \frac{1}{2} d_{\min} s^2 \right) \\
 \le n \exp \left\{ - \left( \frac{1 - \log 2}{2} \right) d_{\min} s^2 \right \} \le n \exp \left ( - \frac{d_{\min} s^2}{7} \right ).
\end{multline*}
Above, the second inequality uses our assumption that $\dmin s^2 \ge 2(k + 1)$,
and the final inequality lower bounds $(1 - \log 2)/2$ by $1/7$. 

Next, in light of the display above, the conclusions of Lemma~\ref{lemma:k_core_expectation_bound} hold when
\begin{equation}
\label{eq:R_condition_1}
\exp \left ( - \frac{\dmin s^2}{7} \right ) \le \frac{1}{5 ( n p_{\max} s^2 )^2 } =  \frac{ ((n - 1)/n)^2}{5 R^2 ( \dmin s^2)^2},
\end{equation}
where the parameter $\gamma$ in Lemma \ref{lemma:k_core_expectation_bound} is taken to be $n p_{\max}$,
and the final equality in the display above follows from the definition of $R$.
To simplify the condition in \eqref{eq:R_condition_1}, we make a few observations. First, for $n$ sufficiently large and $\dmin s^2 \ge R$, \eqref{eq:R_condition_1} is satisfied if 
\begin{equation}
\label{eq:R_condition_2}
\exp \left ( - \frac{\dmin s^2}{7} \right ) \le \frac{ 1}{6 ( \dmin s^2 )^4}.
\end{equation}
Next, it can be seen numerically that $e^{- x/7} \le 1/(6 x^4)$ when $x \ge 154$.
As a result, for $n$ sufficiently large and $\dmin s^2 \ge \max \{ R, 154 \}$, the condition \eqref{eq:R_condition_2} is satisfied, and in light of Lemma \ref{lemma:k_core_expectation_bound}, it follows that $|F_k| \le 3 e^{- \dmin s^2 / 7} + o(n)$.
The bound in~\eqref{eq:core_lower_bound} readily follows.
\end{proof}

\section{Conclusion and discussion}\label{sec:discussion}

In this paper we have initiated the systematic study of graph matching for correlated inhomogeneous random graphs. Our main results show that the $k$-core estimator has desirable information-theoretic guarantees, both in the general setting and for several well-studied specific models.

Our work opens up many questions for future research. 
A specific open question is whether we can relax the assumption on $p_{\max}$ by allowing $p_{\max} \leq n^{-\alpha+o(1)}$ with $\alpha \leq 1/2$. 
A very interesting future direction is to develop 
\emph{model-specific} converse bounds that adapt to the heterogeneity of the underlying model. 
Finally, 
can the recent breakthroughs in computationally efficient estimators for the Erd\H{o}s--R\'enyi case~\cite{mao2021exact,mao2022otter} be extended to the inhomogeneous case?

\bibliographystyle{abbrv}
\bibliography{citations}


\newpage 

\appendices

\section{Proof of Lemma \ref{lemma:k_core_correctness}}\label{sec:proof_k-core_lemma}

For a matching $(M, \mu)$, define 
\[
f(M, \mu, G_1, G_2, \pi_*) : = \sum_{i \in M : \mu(i) \neq \pi_*(i)} \deg_{G_1 \land_\mu G_2}(i).
\]
For brevity, we 
write $f(\mu)$ instead of $f(M, \mu, G_1, G_2, \pi_*)$. In words, $f(\mu)$ denotes the sum of the degrees of vertices in $G_1 \land_{\mu} G_2$ that are incorrectly matched by $\mu$.

\begin{definition}[Weak $k$-core matching]
A matching $(M, \mu)$ is a \emph{weak $k$-core matching} if $f(\mu) \ge k | \{ i \in M : \mu(i) \neq \pi_*(i) \} |$. 
\end{definition}

Stated informally, if $(M,\mu)$ is a weak $k$-core matching, then the average degree in the incorrectly matched region of $\mu$ is at least $k$. 
Notice that if $(M,\mu)$ is a $k$-core matching then it is also a weak $k$-core matching, but the other direction does not necessarily hold. 

We next define the useful notion of a \emph{maximal matching}.

\begin{definition}[$\pi_*$-maximal matching]
A matching $(M, \mu)$ is \emph{$\pi_*$-maximal} if, for every $i \in [n]$, either $i \in M$ or $\pi_*(i) \in \mu(M)$, where $\mu(M)$ is the image of $M$ under $\mu$. Furthermore, we let $\cM(d)$ be the set of $\pi_*$-maximal matchings which make $d$ errors (i.e., $| \{i \in M : \mu(i) \neq \pi_*(i) \}| = d$).
\end{definition}

The following lemma, the proof of which can be found in \cite{GRS22, cullina2020partial}, provides a generic sufficient condition for the correctness of \emph{any} $k$-core matching. 

\begin{lemma}
\label{lemma:generic_correctness}
Let $(G_1, G_2)$ be a pair of random graphs on the vertex set $[n]$ with ground-truth matching $\pi_*$. For any positive integer $k$, define the quantity 
\[
\xi : = \max_{1 \le d \le n} \max_{(M, \mu) \in \cM(d)} \p( f(\mu) \geq kd )^{1/d}.
\]
Let $(\widehat{M}_k, \widehat{\mu}_k)$ be the $k$-core estimator of $G_1$ and $G_2$. Then
\begin{align*}
\p  \left( \widehat{M}_k = \core_k ( G_1 \land_{\pi_*} G_2) \text{ and } \widehat{\mu}_k \{ \widehat{M}_k \} = \pi_* \{ \widehat{M}_k \} \right) &\\
\ge 2 -  \exp & ( n^2 \xi).
\end{align*}
\end{lemma}

In words, $\xi$ is the probability that a given maximal matching is a weak $k$-core matching, normalized by the number of errors made by $(M, \mu)$.
Crucially, Lemma~\ref{lemma:generic_correctness} shows that if $\xi = o(n^{-2})$, then the $k$-core estimator will be correct with high probability. Lemma~\ref{lemma:generic_correctness} was first proved by Cullina, Kiyavash, Mittal, and Poor~\cite{cullina2020partial} for the case of correlated Erd\H{o}s-R\'{e}nyi graphs. It was later noted by Gaudio, R\'{a}cz, and Sridhar~\cite{GRS22} that the proof in~\cite{cullina2020partial} readily extends to generic pairs of correlated random graphs $(G_1, G_2)$. For a proof of Lemma \ref{lemma:generic_correctness}, we defer the reader to \cite[Lemma~19 and Corollary~20]{GRS22}.

We proceed by bounding $\xi$, which is done formally in the following lemma. We remark that the proof closely follows~\cite[Lemma~23]{GRS22}, but we include it here for completeness.

\begin{lemma}
\label{lemma:f_probability_bound}
Let $p_{\max} : = \max_{i,j \in [n]} p_{ij}$.
For any matching $(M, \mu) \in \cM(d)$ and any $\theta > 0$, we have that
\begin{multline*}
\p( f(\mu) \ge kd)  \\
\le 3 \exp \left \{ - d \left( \theta k - e^{2 \theta}  p_{\max} s^2 - n e^{6 \theta} p_{\max}^2 s^2  \right) \right \}.
\end{multline*}
\end{lemma}

\begin{proof}
Define the following sets:
\begin{align*}
\cA(\mu) & : = \{ (i,j) \in M^2 : \mu(i) \neq \pi_*(i) \} \\ 
\cB(\mu) & : = \{ (i,j) \in \cA(\mu) : \mu(i) = \pi_*(j) \text{ and } \mu(j) = \pi_*(i) \} \\
\cC(\mu) & : = \cA(\mu) \setminus \cB(\mu).
\end{align*}
We make a few remarks about these sets. First, since $\mu$ makes $d$ errors, $\cA(\mu) \le d |M| \le dn$. 
Moreover, for $(i,j) \in \cB(\mu)$, $i$ is one of the $d$ vertices misclassified by $\mu$ and there does not exist another $k \in [n]$ such that $(i,k) \in \cB(\mu)$. Hence $|\cB(\mu) | \le d$.

Using $\cA(\mu)$, $\cB(\mu)$, and $\cC(\mu)$, we have the following decomposition of $f(\mu)$:
\begin{align*}
f(\mu) & = \sum_{i \in M : \mu(i) \neq \pi_*(i)} \deg_{G_1 \land_{\mu} G_2} (i)  = \sum_{(i,j) \in \cA(\mu)} A_{ij} B_{\mu(i) \mu(j)} \\
& = \sum_{(i,j) \in \cB(\mu) } A_{ij} B_{\mu(i) \mu(j) } + \sum_{(i,j) \in \cC(\mu) } A_{ij} B_{\mu(i) \mu(j)} \\
& = 2 \sum_{(i,j) \in \cB(\mu) : i < j} A_{ij} B_{\mu(i) \mu(j)} + \sum_{(i,j) \in \cC(\mu)} A_{ij} B_{\mu(i) \mu(j)}
\end{align*}
For brevity, we denote the first summation by $X_{\cB}$ and the second by $X_{\cC}$.
It turns out that $X_{\cB}$ and $X_{\cC}$ are independent.
Indeed, observe that $A_{ij}$ and $B_{\mu(a) \mu(b)}$ are correlated if and only if $\{ \mu(a), \mu(b) \} = \{ \pi_*(i), \pi_*(j) \}$. By the definition of $\cB(\mu)$, we have that $\{ \mu(i), \mu(j) \} = \{ \pi_*(i), \pi_*(j) \}$ for $(i,j) \in \cB(\mu)$, and it follows that the terms of the summation of $X_{\cB}$ are independent of the terms of the summation of $X_{\cC}$.
This same argument shows that the summands comprising $X_{\cB}$ are also independent, with $A_{ij} B_{\mu(i) \mu(j)} = A_{ij} B_{\pi_{*}(i) \pi_{*}(j)}\sim \mathrm{Bern}(p_{ij} s^2)$. It follows that 
\begin{equation}
\label{eq:XB_stochastic_domination}
X_{\cB} \preceq \mathrm{Bin}( | \cB(\mu) | , p_{\max} s^2 ) \preceq \mathrm{Bin} ( d, p_{\max} s^2 ),
\end{equation}
where, in the second stochastic domination, we have used that $|\cB(\mu) | \le d$.

Handling $X_\cC$ is more complicated, since the corresponding summands may be correlated. To get around this issue, we partition $\cC(\mu)$ into $\cC_1(\mu)$, $\cC_2(\mu)$, and $\cC_3(\mu)$, and define 
\[
X_{\cC_m} : = \sum_{(i,j) \in \cC_m(\mu): i < j} A_{ij} B_{\mu(i) \mu(j) }, \qquad m \in \{1,2,3\}.
\]
Notice in particular that $X_{\cC} \le 2 (X_{\cC_1} + X_{\cC_2} + X_{\cC_3})$, 
where the factor of $2$ accounts for the possibility that $(i,j)$ and $(j,i)$ are both elements of~$\cC(\mu)$.

Crucially, we will choose the partition so that for each $m \in \{1,2,3 \}$, $X_{\cC_m}$ is a sum of independent Bernoulli random variables. To this end, consider two unordered pairs $\{i,j \}$ and $\{a,b \}$ in ${M \choose 2}$. The random variables $A_{ij} B_{\mu(i) \mu(j)}$ and $A_{ab} B_{\mu(a) \mu(b)}$ are dependent if and only if one of the following two conditions hold:
\begin{align}
\label{eq:edge_cond_1}
\{ \mu(i) , \mu(j) \} & = \{ \pi_*(a), \pi_*(b) \}, \\
\label{eq:edge_cond_2}
\{ \mu(a), \mu(b) \} & = \{ \pi_*(i), \pi_*(j) \}.
\end{align}
Let us now construct a dependency graph $H$ on the vertex set $V(H) : = \{ \{ i,j \} : (i,j) \in \cC(\mu) \}$ such that $\{i,j \}, \{a,b \} \in V(H)$ have an edge between them if and only if $A_{ij} B_{\mu(i) \mu(j) }$ and $A_{ab} B_{\mu(a) \mu(b)}$ are correlated. Now, since each vertex in $H$ has at most two neighbors in light of the conditions \eqref{eq:edge_cond_1} and \eqref{eq:edge_cond_2}, $H$ is 3-colorable. Letting $\cC_1, \cC_2, \cC_3$ be the partition of $V(H)$ corresponding to the three colors, it follows that $X_{\cC_1}, X_{\cC_2}, X_{\cC_3}$ are sums of independent random variables as desired. 

We proceed by studying the distributions of $X_{\cC_1}, X_{\cC_2}$, and $X_{\cC_3}$. For each $(i,j) \in \cC(\mu)$, we have from the definition of $\cC(\mu)$ that $\{\pi_*(i), \pi_*(j) \} \neq \{ \mu(i), \mu(j) \}$, hence 
\[
A_{ij} B_{\mu(i) \mu(j) } \sim \mathrm{Bern}( p_{ij} p_{ \pi_*^{-1}(\mu(i)), \pi_*^{-1}(\mu(j))} s^2).
\]
In particular, we have that $A_{ij} B_{\mu(i) \mu(j) } \preceq \mathrm{Bern}(p_{\max}^2 s^2)$. Moreover, for $m \in \{1,2,3 \}$,
\begin{equation}
\label{eq:XC_stochastic_domination}
X_{\cC_m} \preceq \mathrm{Bin} ( | \cC_m(\mu)|, p_{\max}^2 s^2 ) \preceq \mathrm{Bin} ( dn, p_{\max}^2 s^2 ),
\end{equation}
where, in the second stochastic domination, we have used that $| \cC_m(\mu) | \le | \cC(\mu) | \le | \cA(\mu) | \le dn$.

We can now bound the probability of interest as follows:
\begin{align*}
\p& ( f(\mu) \ge kd )  \le \p ( 2 X_{\cB} + X_{\cC} \ge kd ) \\
& \stackrel{(a)}{\le} \p ( 2 ( X_{\cB} + X_{\cC_1} + X_{\cC_2} + X_{\cC_3} ) \ge kd ) \\
& \stackrel{(b)}{\le} \sum_{m = 1}^3 \p ( 2 X_{\cB} + 6 X_{\cC_m} \ge kd ) \\
& \stackrel{(c)}{\le} \sum_{m = 1}^3 e^{- \theta kd } \E \left[ e^{2 \theta X_{\cB}} \right] \E \left[ e^{6 \theta X_{\cC_m}} \right] \\
& \stackrel{(d)}{\le} 3 e^{- \theta kd} ( 1 + p_{\max} s^2 (e^{2 \theta } - 1) )^d ( 1 + p_{\max}^2 s^2 ( e^{6 \theta } - 1) )^{dn} \\
& \stackrel{(e)}{\le} 3 \exp \left \{ - d ( \theta k - e^{2 \theta } p_{\max} s^2 - n e^{6 \theta} p_{\max}^2 s^2 ) \right \}.
\end{align*}
In the display above, $(a)$ uses that $X_{\cC} \le 2 ( X_{\cC_1} + X_{\cC_2} + X_{\cC_3})$;
$(b)$ is due to a union bound;
$(c)$ follows for $\theta > 0$ from a Chernoff bound and the independence of $X_{\cB}$ and $X_{\cC_m}$;
$(d)$ is obtained by bounding the moment generating functions of $X_{\cB}$ and $X_{\cC}$, using that $X_{\cB}$ and $X_{\cC_m}$ can be stochastically dominated by binomial random variables (see \eqref{eq:XB_stochastic_domination} and \eqref{eq:XC_stochastic_domination});
$(e)$ uses the inequality $(1 + x)^t \le e^{tx}$.
\end{proof}

We are now ready to prove Lemma~\ref{lemma:k_core_correctness}, which is the main goal of this section.

\begin{proof}[Proof of Lemma~\ref{lemma:k_core_correctness}]
Set $\theta : = c \log n$ for some $2/k < c < (2 \alpha - 1) / 6$. 
Such a $c$ exists due to the assumption that $k > 12/(2\alpha - 1)$. 
With this choice of $\theta$, we have that
\[
e^{2 \theta}  p_{\max}s^2 \le s^2 n^{- ( \alpha + 1)/3 + o(1)} = o(1).
\]
We also have that 
\[
n e^{6 \theta}  p_{\max}^2s^2 \le s^2 n^{6c - (2 \alpha - 1) + o(1)} = o(1),
\]
which follows since $6c < 2 \alpha - 1$. Lemma~\ref{lemma:f_probability_bound} now shows that 
\[
\p( f(\mu) \ge kd ) \le 3 \exp \left \{ - d \left( ck \log n - o(1) \right) \right \},
\]
which in turn implies that $\xi \le 3 e^{- ck \log n + o(1)}$. 
Since we chose $c$ such that $ck > 2$, we have that $\xi = o(n^{-2})$, 
which proves the desired result in light of Lemma~\ref{lemma:generic_correctness}.
\end{proof}

\section{Proofs of corollaries}\label{sec:proofs_corollaries}

Throughout this section, we denote $\dmin : = \min_{i \in [n]} \sum_{j = 1}^n p_{ij}$.

\subsection{Proofs for exact graph matching}

\begin{proof}[Proof of Corollary~\ref{cor:sbm_exact}]
Let $a \in [m]$, and suppose that $i \in V_a$. For $n$ sufficiently large, it holds that
\begin{multline}
\label{eq:dmin_sbm}
\dmin s^2 = \sum_{b = 1}^m |V_b| q_{ab} s^2 - q_{aa}s^2 \\
\ge (1 + \epsilon ) \log n - 1 \ge (1 + \epsilon/2) \log n. 
\end{multline}
In the first inequality above, we have used that $q_{aa}s^2 \le 1$.
An application of Theorem \ref{thm:exact_recovery} proves the corollary.
\end{proof}

\begin{proof}[Proof of Corollary~\ref{cor:cl_exact}]
We start by showing that $p_{\max} \le n^{- \alpha + o(1)}$ if \eqref{eq:chung_lu_exact_recovery_1} holds.
For any $i,j \in [n]$,
\begin{equation}
\label{eq:chung_lu_probability_bound}
\frac{w_i w_j }{\sum_{i = 1}^n w_i } \le  \left( \frac{\max_{i \in [n]} w_i}{ \sqrt{\sum_{i = 1}^n w_i }} \right)^2 \le n^{- \alpha + o(1)},
\end{equation}
where the final bound is a consequence of \eqref{eq:chung_lu_exact_recovery_1}.
In particular, since $p_{ij} = w_i w_j / ( \sum_{k = 1}^n w_k )$ for distinct $i,j \in [n]$, it follows that $p_{\max} \le n^{- \alpha + o(1)}$.
It remains to lower bound the minimum expected degree. For any $i \in [n]$, it holds for $n$ sufficiently large that
\begin{align}
\label{eq:dmin_chung_lu}
\dmin s^2 &  = \min_{i \in [n]} \sum_{j \in [n] : j \neq i} \frac{ w_i w_j s^2 }{\sum_{k = 1}^n w_k } \nonumber \\
& \ge \min_{i \in [n]} \sum_{j = 1}^n \frac{w_i w_j s^2}{\sum_{k = 1}^n w_k } - p_{\max} s^2 \nonumber \\
& \ge \min_{i \in [n]} w_i s^2 - n^{ - \alpha + o(1)} \ge (1 + \epsilon /2) \log n.
\end{align}
In the display above, the first inequality uses \eqref{eq:chung_lu_probability_bound}.
The desired result now follows from Theorem \ref{thm:exact_recovery}.
\end{proof}

\begin{proof}[Proof of Corollary~\ref{cor:rgg_exact}]
From \eqref{eq:rg_probability}, it is evident that $p_{\max} \leq p$, hence $p_{\max} \le n^{- \alpha + o(1)}$. The desired result follows by writing out the condition~\eqref{eq:expected_degree_condition} from Theorem \ref{thm:exact_recovery} in this setting. 
\end{proof}

\subsection{Proofs for partial graph matching}

\begin{proof}[Proof of Corollary~\ref{cor:sbm_partial}]
Since $\dmin \ge (n - 1) q_{\min}$, we have that $R \le q_{\max} / q_{\min}$. Furthermore, we have that
\begin{multline*}
\dmin s^2 = \min_{a \in [m]} \sum_{b = 1}^m |V_b| q_{ab} s^2 - o(1) \\
\ge  \max \left \{ \frac{q_{\max}}{q_{\min}} , 154 \right \} \ge \max \{ R, 154 \},
\end{multline*}
where the equality in the display above follows from \eqref{eq:dmin_sbm}.
The desired result follows by applying Theorem \ref{thm:partial_matching}.
\end{proof}

\begin{proof}[Proof of Corollary~\ref{cor:cl_partial}]
We start by upper bounding $R$ as follows:
\[
R \le \frac{p_{\max}}{p_{\min}} \le \left( \frac{w_{\max}^2}{\sum_{k = 1}^n w_k } \right) \left( \frac{ w_{\min}^2}{\sum_{k = 1}^n w_k} \right)^{-1} = \left( \frac{ w_{\max}}{w_{\min}} \right)^2.
\]
Furthermore, we have that
\begin{align*}
\dmin s^2 & = w_{\min} s^2 - o(1) \\
& \ge  \max \left \{ \left( \frac{w_{\max}}{w_{\min}} \right)^2 , 154 \right \}  \\
& \ge \max \left \{ R, 154 \right \},
\end{align*}
The desired result now follows from applying Theorem \ref{thm:partial_matching}.
\end{proof}

\begin{proof}[Proof of Corollary \ref{cor:rgg_partial}]
As a shorthand, denote $D : = \min_{i \in [n]} | \{ j \in [n] \setminus \{i \}: \| x_i - x_j \| \le r \} |$.
Noting that $p_{\max} \le p$ for this model, we have the upper bound $R \le n / D$. Noting that $\dmin = p D$, the conditions of Theorem \ref{thm:partial_matching} are satisfied if $ps^2 D \ge \max \{ n/D, 154 \}$, or equivalently, if $ps^2 D \ge \max \{ s \sqrt{np}, 154\}$.
Finally, since $p s^2 D = \min_{i \in [n]} | \{ j \in [n] : \| x_i - x_j \| \le r \}| - o(1)$, the desired result readily follows.
\end{proof}

\section{Proofs of lemmas on degrees}\label{sec:proofs_lemmas_degrees}

\begin{proof}[Proof of Lemma~\ref{lemma:Zk_expectation}]
We begin by writing $\E [ |Z_k| ] = \sum_{i = 1}^n \p ( \deg_G(i) \le k )$; we proceed by bounding the terms of the summation. For any $i \in [n]$, we have the distributional representation $\deg_G(i) \stackrel{d}{=} \sum_{j = 1}^n Y_j$, where the $Y_j$'s are independent and $Y_j \sim \mathrm{Bern}(p_{ij})$.
For $\theta > 0$, it holds by Markov's inequality that
\begin{align*}
\p & ( \deg_{G}(i) \le k )  \le e^{\theta k} \E \left[ \exp ( - \theta \deg_{G}(i) ) \right] \\
& = e^{\theta k} \prod_{j = 1}^n \E [ \exp ( - \theta Y_j )] = e^{\theta k} \prod_{j = 1}^n \left( 1 + p_{ij}  (e^{- \theta} - 1) \right) \\
& \le \exp \left \{ \theta k - (1 - e^{- \theta}) \sum_{j = 1}^n p_{ij}  \right \}.
\end{align*}
The desired claim follows from setting $\theta = - \log (1 - c)$. 
\end{proof}

\begin{proof}[Proof of Lemma~\ref{lemma:Zk_concentration}]
Our strategy is to bound the variance of $|Z_k|$ and apply Chebyshev's inequality. Let $E_i := \mathbf{1}\left( \deg_G(i) \le k\right)$, so that $|Z_k| = \sum_{i = 1}^n E_i$, and
\[
\mathrm{Var}(| Z_k | ) = \sum_{i = 1}^n \mathrm{Var}(E_i) + \sum_{i,j \in [n] : i \neq j} \mathrm{Cov}(E_i, E_j).
\]
As $E_i$ is an indicator variable, we can bound $\mathrm{Var}(E_i) \le \E [ E_i ]$, which allows us to bound the first summation by $\E [ | Z_k | ]$. 
We now turn to the terms of the second summation. Defining the indicator variable $A_{ij} : = \mathbf{1}( (i,j) \in E(G))$, 
we have that 
\begin{align}
\E [ E_i E_j ] & = \E \left[ \E [ E_i \vert A_{ij} ] \E [ E_j \vert A_{ij}] \right] \nonumber \\
& \stackrel{(a)}{\le} p_{ij} + (1 - p_{ij}) \E [ E_i \vert A_{ij} = 0] \E [ E_j \vert A_{ij} = 0] \nonumber  \\
& \stackrel{(b)}{\le} p_{ij} + (1 - p_{ij} ) \frac{ \E [ E_i ] \E [ E_j ] }{ (1 - p_{ij} )^2} \nonumber \\
\label{eq:Ei_product_expectation}
& \le p_{ij} + ( 1 - p_{ij} )^{-1} \E [ E_i ] \E [ E_j ].
\end{align}
Above, $(a)$ follows since $A_{ij} = 1$ with probability $p_{ij}$ and $|E_i|, |E_j | \le 1$; 
$(b)$ uses the relation
\[
\E [ E_i \vert A_{ij} = 0] = \frac{ \p ( E_i = 1, A_{ij} = 0 ) }{ \p ( A_{ij} = 0)} \le \frac{\E [ E_i ] }{1 - p_{ij} }.
\]
Using \eqref{eq:Ei_product_expectation}, we can bound $\mathrm{Cov}(E_i, E_j)$ as 
\begin{align}
\mathrm{Cov}(E_i, E_j ) & = \E [ E_i E_j ] - \E [ E_i ] \E [ E_j ] \nonumber \\
& \le p_{ij} + \left ( (1 - p_{ij} )^{-1} -1  \right) \E [ E_i ] \E [ E_j ] \nonumber \\
& \le p_{ij} + 2 p_{ij} \E[ E_i ] \E [ E_j ] \nonumber \\
\label{eq:Ei_cov_bound}
& \le 3 p_{ij} \le 3 p_{\max}.
\end{align}
Above, the inequality on the third line holds since $(1 - x)^{-1} - 1 = x/(1-x)\le 2x$ for $x$ sufficiently small, and the first inequality on the third line holds since $0 \le E_i, E_j \le 1$.
Putting everything together, it follows that 
\begin{equation}\label{eq:variance_estimate}
\mathrm{Var}( | Z_k | ) \le \E [ | Z_k | ] + 3  n^2 p_{\max} = o(n^{3/2}),
\end{equation}
where the first inequality is due to the covariance bound in~\eqref{eq:Ei_cov_bound}. The final $o(n^{3/2})$ bound follows since $|Z_k| \le n$ and we assumed that $p_{\max} = o(1/\sqrt{n})$.
With the variance estimate~\eqref{eq:variance_estimate} in hand, Chebyshev's inequality implies 
the claim. 
\end{proof}

\section{Proof of the \Luczak expansion lemma}\label{sec:Luczak}

\begin{proof}[Proof of Lemma~\ref{lemma:k_core_size_1}]
The proof follows using similar methods as \Luczak~\cite{Luczak1991}.
Define a sequence of subsets of vertices $\{ U_m \}_{m \ge 0}$ as follows. Let $U_0 := Z_{k+1}$. For $m \ge 0$, if there exists $v \in V \setminus U_m$ with at least $3$ neighbors in $U_m$, then let $U_{m + 1} : = U_m \cup \{ v \}$. If no such vertex exists, we stop the construction. 

Let $U_\ell$ be the final set of the construction, and let us assume by way of contradiction that $| U_\ell| > 3 | U_0|$. Then there must exist $0 \le b \le \ell$ such that $|U_b| = 3 | U_0|$, as only a single vertex is added per iteration of the construction. Letting $E_b$ denote the number of edges in $G \{ U_b \}$, we have that 
\[
E_b \ge 3 ( | U_b | - | U_0 | ) = 2 | U_b|.
\]
Since $| U_b | = 3 | U_0 | \le \frac{3}{4 \gamma^2 } n$, the display above contradicts the event in Lemma \ref{lemma:subgraph_density}. Hence $| U_\ell | \le 3 | U_0 | = 3 |Z_{k+1}|$. 

Finally, we connect this analysis to the $k$-core of $G$ by noting that, if $v \in V \setminus U_\ell$, then $\deg_G(v) \ge k + 2$ and $v$ has at most $2$ neighbors in $U_\ell$. As a result, the minimum degree in $G \{ V \setminus U_\ell \}$ is at least $k$, 
so 
$
F_{k} \subseteq U_\ell$.
\end{proof}

\end{document}